\documentclass[a4paper, 11pt]{amsart}

\usepackage{enumerate}
\usepackage{mathrsfs}
\usepackage{amssymb,amsmath}
\usepackage[all]{xy}
\usepackage{color}

\title[]{On global Okounkov bodies of spherical varieties}
\author{Guido Pezzini and Henrik Sepp\"anen} 

\thanks{The first author was supported by DFG grant BA 1559/6-1. The second author was partially supported by 
the DFG Priority Programme 1388 ``Representation Theory"}

\address{Guido Pezzini,
Dipartimento di Matematica ``Guido Castelnuovo'',
 ``Sapienza'' Universit\`a di Roma,
Piazzale A.\ Moro 5,
I-00185 Roma,
Italy}
\email{pezzini@mat.uniroma1.it}

\address{Henrik Sepp\"{a}nen,
Mathematisches Institut,
Georg-August-Univer\-sit\"at G\"ot\-tingen,
Bunsenstra\ss e 3-5, 
D-37073 G\"ottingen,
Germany}
\email{henrik.seppaenen@mathematik.uni-goettingen.de}

\newcommand{\C}{\mathbb{C}}
\newcommand{\R}{\mathbb{R}}
\newcommand{\N}{\mathbb{N}}

\newcommand{\Q}{\mathbb{Q}}

\renewcommand{\phi}{\varphi}

\renewcommand{\tilde}{\widetilde}

\newcommand{\ord}{\mbox{ord}}

\renewcommand{\to}{\longrightarrow}
\renewcommand{\O}{\mathcal O}

\newtheorem{prop}{Proposition}[section]
\newtheorem{lemma}[prop]{Lemma}
\newtheorem{cor}[prop]{Corollary}
\newtheorem{thm}[prop]{Theorem}
\theoremstyle{definition}
\newtheorem{defin}[prop]{Definition}

\newtheorem{remark}[prop]{Remark}

\begin{document}

\begin{abstract}
We define and study the {\em global Okounkov moment cone} of a projective spherical variety $X$, generalizing both the global Okounkov 
body and the {\em moment body} of $X$ defined by Kaveh and Khovanskii. Under mild assumptions on $X$ we show that the global Okounkov moment 
cone of $X$ is rational polyhedral. As a consequence, also the global Okounkov body of $X$, with respect to a particular valuation, is 
rational polyhedral.
\end{abstract}
 
\maketitle

\section*{Introduction}

Started by the work of A.\ Okounkov, cf.\ \cite{Ok96}, the theory of Okounkov bodies has been widely developed in recent years, in connection 
with other research topics such as representation theory, asymptotics of linear series and positivity, and tropical geometry, 
see e.g.\ \cite{LM09}, \cite{KK09}, \cite{KL15}, \cite{Se16}, \cite{SS17}, \cite{PU16}.

In this short note, we study a convex cone associated to a projective variety $X$, spherical under the action of a semisimple and simply 
connected group $G$. The convex cone, called {\em global Okounkov moment cone} and denoted by%
\footnote{The meaning of $\nu$ in this notation is recalled below.} $\widetilde \Delta_\nu(X)$, is closely related to the {\em global Okounkov body}, 
which we denote here by $\Delta_\nu(X)$. The latter is a ``global version'' of the Okounkov bodies associated to pseudo effective divisors of projective varieties. 

We recall that the construction of $\Delta_\nu(X)$ depends on the choice of a function
\[
\nu: \bigsqcup_{D \in \text{Pic}(X)} H^0(X, \mathcal{O}_X(D)) \setminus \{0\} \rightarrow \N_0^n,
\]
where $n=\dim X$, called a {\em valuation-like} function.

In the case of a general smooth projective variety $X$, such function is often constructed using a flag of 
subvarieties $X=Y_0 \supset Y_1 \supset \ldots \supset Y_n = \{\text{pt}\}$, where $Y_i$ is a smooth prime divisor of $Y_{i-1}$ for all $i\in \{1,\ldots,n\}$. For $D$ a 
Cartier divisor and $s\in H^0(X, \mathcal{O}_X(D))$, one defines $\nu(s)=(\nu_1(s),\ldots,\nu_n(s))$ setting first $\nu_1(s)=\ord_{Y_1}(s)$. Then, by restriction 
on $Y_1$, the section $s$ determines naturally a section $s_1\in H^0(Y_1,\mathcal O_{Y_1}(D-\nu_1(s)Y_1))$, and we set $\nu_2(s) = \ord_{Y_2}(s_2)$. Iterating this 
construction defines $\nu_i(s)$ for all $i\in\{1,\ldots,n\}$. Other definitions of valuation-like functions are possible, and sometimes more convenient. This is the 
case e.g.\ for flag varieties, where Schubert varieties are natural candidates for the elements of the flag, but the possible failure of their smoothness leads 
to alternative constructions. One such, which uses Bott-Samelson varieties, is described in \cite{SS17}.

Once the function $\nu$ is given, the global Okounkov body is defined as the closed convex cone in $\R^n\times N^1(X)$ generated by all pairs 
of the form $(\nu(s),[D])\in \R^n\times N^1(X)_\R$ where $D$ is effective and $s$ is as above.

In the setting of spherical varieties it is convenient to use valuation-like functions defined using a mixture of the two constructions 
mentioned above, see Section~\ref{s:globalOkounkovmomentcone} below for details.

In addition, given the importance of the structure of $G$-module of the space of sections $H^0(X, \mathcal{O}_X(D))$, it is natural to add to the points of the 
global Okounkov body a third component. We assume that $s$ belongs to a simple submodule $V$ of $H^0(X, \mathcal{O}_X(D))$ (we show in Section~\ref{s:globalOkounkovmomentcone} 
that this assumption is harmless), and, instead of the pair $(\nu(s),[D])$, we consider the triple $(\nu(s),[D], \lambda)$ where $\lambda$ is the highest weight of $V$. The 
closed convex cone generated by such triples is the global Okounkov moment cone $\widetilde \Delta_\nu(X)$, see Definition~\ref{def:gOmc} below. This construction 
can also be considered a generalization of the {\em moment body} of $X$ defined by Kaveh-Khovanskii 
in \cite{KK12}.

Our main result is Theorem~\ref{T:polyhedral} below, where we show under mild hypotheses on $X$ that $\widetilde \Delta_\nu(X)$ is rational polyhedral. This also 
implies that $\Delta_\nu(X)$ is rational polyhedral, answering in our setting a question posed in \cite{LM09}.

We end this introduction by remarking that our varieties are Mori dream spaces (see \cite[Theorem~4.1.1]{perrin}), as defined by Hu and Keel (\cite{hk}). As a consequence, they also admit valuations defining rational 
polyhedral Okounkov bodies by the method of Postinghel and Urbinati (\cite{PU16}) that uses tropical compactifications. It is not clear to us, however, whether our 
valuation can be related to theirs.

\section{Setting}\label{s:setting}
Let $G$ be a simply connected semisimple complex algebraic group, let $B \subseteq G$ be a Borel subgroup and $T\subseteq B$ a maximal torus. We denote 
by $W$ the Weyl group of $G$. The Lie algebras of $G,B$, and $T$ will be denoted by the corresponding lower case German letter. Given a 
dominant weight $\lambda$, we denote by $V(\lambda)$ the irreducible $G$-module of highest weight $\lambda$.

We recall that an irreducible normal $G$-variety is {\em spherical} if $B$ has an open orbit on $X$.

Throughout the paper $X$ will be a projective $\Q$-factorial spherical $G$-variety of dimension $n$. We shall further assume that $X$ admits $G$-invariant Cartier prime divisors 
$E_1,\ldots, E_r$ with $E_i=Z(s_i)$, for a $G$-invariant section $s_i \in H^0(X, \mathcal{O}_X(E_i))$, with combinatorial normal crossings such that 
the intersection $E_1 \cap \cdots \cap E_r$ is a closed $G$-orbit $Gx_0$. Then $Gx_0 \cong G/P$ for a parabolic subgroup $P$ containing $B$. 
We assume that all intersections of the $E_i$ are smooth along the orbit $Gx_0$.

An example where these assumptions are satisfied is the case where $X$ is a {\em wonderful variety}, see e.g.\ \cite{Pe17}.

We also consider a $B$-equivariant birational morphism 
\begin{align}
\tau: Z \to G/P, \label{E: bsres}
\end{align}
where $Z$ is a Bott-Samelson variety defined by a reduced sequence 
$w=(s_{i_1}, \ldots, s_{i_m})$, where $s_{i_j} \in W$ are simple reflections (cf.\ e.g.\ \cite[Section 13.10]{jantz}).

We recall that any line bundle on $X$ has a unique $G$-linearization, by \cite{knop-kraft-luna-vust}.

\section{The global Okounkov moment cone}\label{s:globalOkounkovmomentcone}

We define a partial flag of subvarieties 
\begin{align*}
 X:=Y_0 \supset Y_1 \supset \cdots \supset Y_r=Gx_0,
\end{align*}
of $X$ by 
\begin{align*}
 Y_i:=\bigcap_{j=1}^i E_i, 
\end{align*}
i.e., $Y_i=Z(s_1,\ldots, s_i)$, for $i=1,\ldots, r$.

We now define a valuation-like function with values in $\N_0^n$ on the disjoint union 
$\bigsqcup_{D \in \text{Pic}(X)} H^0(X, \mathcal{O}_X(D)) \setminus \{0\}$ 
of all section spaces of all line bundles
by first using the $\N_0^r$-valued function defined by the partial flag above, and then continuing with the valuation-like function 
defined by the ``vertical flag'' on a Bott-Samelson variety (cf. \cite{SS17}). 

\noindent We thus first define 
\begin{align*}
 \nu': \bigsqcup_{D \in \text{Pic}(X)} H^0(X, \mathcal{O}_X(D)) \setminus \{0\} \rightarrow \N_0^r
\end{align*}
by $\nu'(s):=(a_1,\ldots, a_r)$, for a Cartier divisor $D$ and nonzero section $s \in H^0(X, \mathcal{O}_X(D))$, where 
\begin{align}
 a_1&:=\mbox{ord}_{Y_1}(s), \nonumber\\
 a_2&:=\mbox{ord}_{Y_2}\left(\frac{s}{s_1^{a_1}}\mid_{Y_1}\right),\nonumber \\
 &\vdots \nonumber \\
 a_r&:=\mbox{ord}_{Y_r}\left(\frac{s}{s_1^{a_1} \cdots s_{r-1}^{a_{r-1}}}\mid_{Y_{r-1}}\right). \label{E: partialval}
\end{align}
Recall the valuation-like function 
\begin{align*}
 \tilde{\nu''}: \bigsqcup_{D \in \text{Pic}(Z)} H^0(Z, \mathcal{O}_Z(D)) \setminus \{0\} \rightarrow \N_0^m
\end{align*}
defined by the ``vertical flag'' of Bott-Samelson-subvarieties of $Z$ with respect to the reduced expression $w=(s_{i_1},\ldots, s_{i_m})$.
This is a flag of Bott-Samelson varieties $Z_j$, where $Z_j$ is defined by the reduced subsequence $(s_{i_{j+1}},\ldots, s_{i_m})$ of $w$ (cf. \cite{SS17}).
Using the birational map $\tau: Z \to G/P$ (eq. \eqref{E: bsres}), we then define
\begin{align*}
\nu'':& \bigsqcup_{D \in \text{Pic}(G/P)} H^0(G/P, \mathcal{O}_{G/P}(D)) \setminus \{0\} \rightarrow \N_0^m,\\
\quad \nu''(s)&:=\tilde{\nu''}(\tau^*s), \quad s \in H^0(G/P, \mathcal{O}_{G/P}(D)).
\end{align*}
By putting $(a_{r+1}, \ldots, a_n):=\nu''(\frac{s}{s_1^{a_1} \cdots s_r^{a_r}}\mid_{Y_r})$ (cf. \eqref{E: partialval}), and defining 
$\nu(s):=(a_1,\ldots, a_n)$, we obtain a valuation-like function 
\begin{align}
 \nu: \bigsqcup_{D \in \text{Pic}(X)} H^0(X, \mathcal{O}_X(D)) \setminus \{0\} \rightarrow \N_0^n.
\end{align}
We write $\nu(s)=(\nu_1(s),\ldots, \nu_n(s))$, for a section $s$ of some line bundle.

One defines similarly a valuation-like function $\nu_{Y_i}$ for sections of line bundles on $Y_i$, for any variety $Y_i$ belonging to the flag $Y_\bullet$.

\begin{remark}
We recall that, since $X$ is spherical, for any line bundle $\mathcal L$ on $X$ the space of sections $H^0(X,\mathcal L)$ is a 
multiplicity-free $G$-module, that is, any simple 
submodule appears with multiplicity $1$. This fact is used in the proof of the following proposition.
\end{remark}

\begin{prop} \label{P: sepvalues}
Let $D$ be an effective divisor on $X$, and consider the decomposition 
 \begin{align*}
  H^0(X, \O_X(D))=\bigoplus_{\lambda} V(\lambda)
 \end{align*}
of the section space $H^0(X, \O_X(D))$ into a direct sum of irreducible $G$-modules.\\

\noindent (i) For any $V(\lambda) \subseteq H^0(X, \O_X(D))$, and any two nonzero sections $s, t \in V(\lambda)$, we have
$\nu'(s)=\nu'(t)$.\\

\noindent (ii) 
If $s \in V(\lambda), t \in V(\mu)$ are nonzero sections for 
distinct highest weights $\lambda$ and $\mu$, then 
$\nu'(s) \neq \nu'(t)$. In particular, $\nu(s) \neq \nu(t)$.
\end{prop}

\begin{proof}
For the first part, note that the $G$-invariance of the $Y_i$ implies that $\nu'$ is $G$-invariant, i.e., $\nu'(gs)=\nu'(s)$, for any 
nonzero $s \in H^0(X, \O_X(D))$, and $g \in G$. Moreover, $\nu'(s+t) \geq \mbox{min}\{\nu'(s), \nu'(t)\}$, for any 
nonzero $s, t \in H^0(X, \O_X(D))$ such that $s+t \neq 0$. 

Now, if $s \in V(\lambda) \in H^0(X, \O_X(D))$ is a nonzero section, it is a cyclic vector for the $G$-module. Hence, 
if $t \in V(\lambda) \subseteq H^0(X, \O_X(D))$, we can write $t$ as 
\begin{align*}
 t=\sum_{i=1}^m c_ig_is,
\end{align*}
for some $g_1,\ldots, g_m \in G$, and $c_1,\ldots, c_m \in \C$. The $G$-invariance and the valuation-like properties of $\nu'$
then yield 
\begin{align*}
\nu(t) \geq \mbox{min}\{\nu(g_is)\mid i=1,\ldots, m\}=\nu(s).
\end{align*}
Reversing the roles of $s$ and $t$, that is, using $t$ as a $G$-cyclic vector, yields $\nu(s) \geq \nu(t)$. This proves (i).\\

For (ii), assume that $\nu'(s)=\nu'(t)=(a_1,\ldots, a_r)$, for nonzero sections $s \in V(\lambda), t \in V(\mu)$, of $\O_X(D)$. 
By (i), we may without loss of generality assume that 
$s$ and $t$ are $B$-semiinvariants of weight $\lambda$ and $\mu$, respectively.
 
 The definition of $\nu'$, and the fact that $Y_i$ is the common zero set of the $G$-invariant sections $s_1,\ldots, s_i$, with $s_i \in H^0(X, \O_X(E_i))$, shows that 
 the restriction $\frac{s}{s_1^{a_1} \cdot \cdots \cdot s_r^{a_r}}\mid_{G/P}$ defines a regular section of the line bundle $\O_{G/P}(D-(a_1E_1+\cdots+a_rE_r))$ over $G/P$. 
 Moreover, since $s$ is $B$-semiinvariant of weight $\lambda$, and the $s_i$ are $G$-invariant, the restricted section $\frac{s}{s_1^{a_1} \cdot \cdots \cdot s_r^{a_r}}\mid_{G/P}$ is 
 also $B$-semiinvariant of weight $\lambda$. The same argument, applied to $t$, shows that 
 $\frac{t}{s_1^{a_1} \cdot \cdots \cdot s_r^{a_r}}\mid_{G/P}$ is a $B$-semiinvariant section of $\O_{G/P}(D-(a_1E_1+\cdots+a_rE_r))$ of weight $\mu$. 
 Since $H^0(G/P, \O_{G/P}(D-(a_1E_1+\cdots+a_rE_r)))$ defines an irreducible $G$-module, this is a contradiction unless $\lambda \neq \mu$. 
 \end{proof}

\begin{defin}
 Let 
 \begin{align*}
  \widetilde S_\nu(X):=\{(\nu(s), [D], \lambda) \in \N_0^n \times N^1(X) \times \Gamma \;\mid\; & s \in V(\lambda)\setminus \{0\}, \\
  & V(\lambda) \subseteq H^0(X, \O_X(D)) \}.
 \end{align*}

\end{defin}

\begin{prop}
 The set $\widetilde S_\nu(X)$ is an additive subsemigroup of $\N_0^n \times N^1(X) \times \Gamma$.
\end{prop}

\begin{proof}
 Let $s \in V(\lambda) \subseteq H^0(X, \O_X(D))$ and $t \in V(\mu) \subseteq H^0(X, \O_X(E))$ be nonzero sections. We show that $\nu(s)+\nu(t)$ is the 
 value of a section in the $G$-submodule $V(\lambda+\mu) \subseteq H^0(X, \O_X(D+E))$. 
 
 Since $\nu(s)=\nu(s+cv_\lambda)$ for all $c \in \C$, where $v_\lambda \in V(\lambda)$ is a $B$-semiinvariant section, we may assume that the decomposition
 of $s$ as a sum of weight vectors has a nonzero component of weight $\lambda$. Similarly, we may assume that $t$ has a nonzero weight component of weight $\mu$. 
 The product section $s\cdot t \in H^0(X, \O_X(D+E))$ decomposes as a sum
 \begin{align}
  s \cdot t=\sum_\xi r_\xi, \quad \xi \in V(\xi) \subseteq H^0(X, \O_X(D+E)). 
  \label{E: decomprod}
 \end{align}
Since $\nu(r_\xi) \neq \nu(r_\xi')$ for $\xi \neq \xi'$, by Proposition \ref{P: sepvalues}, we have that 
\begin{align*}
 \nu(s \cdot t)=\mbox{min} \{\nu(r_\xi)\}, 
\end{align*}
where the minimum is taken over the set of highest weights occurring in the decomposition \eqref{E: decomprod}.  
In fact, this minimum is attained by the highest weight $\xi_0$ for which $\nu'(s \cdot t)=\nu'(r_{\xi_0})$. Since $\nu'(s \cdot t)=\nu'(s)+\nu'(t)=\nu'(v_\lambda)+\nu'(v_\mu)=\nu'(v_\lambda \cdot v_\mu)
=\nu'(v_{\lambda+\mu})$, it follows that $\xi_0=\lambda+\mu$. 
Hence, 
\begin{align*}
 (\nu(s), [D], \lambda)+(\nu(t), [E], \mu)=(\nu(r_{\lambda+\mu}), [D]+[E], \lambda+\mu) \in \widetilde S_\nu(X), 
\end{align*}
where $r_{\lambda+\mu}$ is the component in $V(\lambda+\mu)$ of the product $$s \cdot t \in H^0(X, \O_X(D+E)).$$ 
This proves that $\widetilde S_\nu(X)$ is a semigroup.
\end{proof}

\begin{defin}
The {\em global Okounkov semigroup} of $X$ (with respect to $\nu$) is
\begin{align*}
 S_\nu(X):=\{(\nu(s), [D]) \in \N_0^n \times N^1(X) \mid s \in H^0(X, \O_X(D)) \setminus \{0\}\}.
 \end{align*}
The {\em global Okounkov body} of $X$ (with respect to $\nu$) is the closed convex cone $\Delta_\nu(X)=\overline{\mbox{cone}}(S_\nu(X))$ generated by $S_\nu(X)$.
\end{defin}

\begin{prop}
The map $p\colon \widetilde S_\nu(X) \to S_\nu(X)$ given by $(\nu(s), [D], \lambda) \mapsto (\nu(s), [D])$ defines an isomorphism of semigroups. 
\end{prop}

\begin{proof}
 Clearly, $p$ is a homomorphism of semigroups. For the surjectivity, let $(\nu(s), [D]) \in S_\nu(X)$, so that $s \in H^0(X, \O_X(D))$. The section $s$ then 
 decomposes as a sum $s=\sum_\xi r_\xi$, where $r_\xi \in V(\xi) \subseteq H^0(X, \O_X(D))$. Since, by Proposition \ref{P: sepvalues}, the values of the $r_\xi$ are distinct, we 
 have $\nu(s)=\nu(r_\lambda)$, for some $\lambda \in \Gamma$ with $V(\lambda) \subseteq H^0(X, \O_X(D))$. Then, $(\nu(r_\lambda), [D], \lambda) \in \widetilde S_\nu(X)$, and 
 \begin{align*}
  p(\nu(r_\lambda), [D], \lambda)=(\nu(r_\lambda), [D])=(\nu(s), [D]).
 \end{align*}
 
 For the injectivity, let $(\nu(s), [D], \lambda)$ and $(\nu(s'), [D'], \lambda')$ be in $\widetilde S_\nu(X)$ with $\nu(s)=\nu(s')$ and $[D]=[D']$. Then $\O_X(D)\cong\O_X(D')$, and by 
 Proposition \ref{P: sepvalues} we have $\lambda=\lambda'$.
\end{proof}

\begin{defin}\label{def:gOmc}
 Let $\widetilde \Delta_\nu(X):=\overline{\mbox{cone}}(\widetilde S_\nu(X)) \subseteq \R^n \oplus N^1(X)_\R \oplus \mathfrak{t}^*$ be the closed convex cone generated by the 
 semigroup $\widetilde S_\nu(X)$.
  We call $\widetilde \Delta_\nu(X)$ the \emph{global Okounkov moment cone} of $X$ (with respect to $\nu$). \\
 \noindent Similarly, for any variety $Y_i$ in the flag $Y_\bullet$, we define $\widetilde \Delta_{\nu_{Y_i}}(Y_i) \subseteq \R^{n-i} \oplus N^1(Y_i)_\R \oplus \mathfrak{t}^*$ to be the closed convex cone generated by the 
 corresponding semigroup $\widetilde S_{\nu_{Y_i}}(Y_i)$ for the spherical variety $Y_i$.
 \end{defin}

  \begin{remark}
  Since the semigroup $S_\nu(X)$ generating $\Delta_\nu(X)$ also keeps track of the $G$-modules from which sections with given values 
  come, the cone $\Delta_\nu(X)$ can also be seen as a global version of the moment polytope of $X$ (cf. \cite{brion-applicmoment}), with the additional information of the values of $\nu$; hence the name.
 \end{remark}

 Let $Y:=Y_1$ be the divisor of the flag $Y_\bullet$. 
 The goal now is to prove that the convex cone $\widetilde \Delta_\nu(X)$ is rational polyhedral if $\widetilde \Delta_{\nu_Y}(Y)$ is rational polyhedral.
 
 We first observe that, for an effective divisor $D$ on $X$ and a $G$-submodule $V(\lambda) \subseteq H^0(X, \O_X(D))$, because of the $G$-equivariance, the 
 restriction map
 \begin{align}
  V(\lambda) \to H^0(Y, \O_Y(D)), \quad s \mapsto s\mid_Y, \label{E: equivrestr}
 \end{align}
is either zero or injective. This motivates the following definition.
\begin{defin}
\noindent (i) 
Let $J\subseteq \overline{\mbox{Eff}}(X) \times \mathfrak{t}^*$ be the subset of all pairs $(D, \lambda)$ with $D$ and $\lambda$ as above, and such that the restriction map \eqref{E: equivrestr} is 
injective, and define $K \subseteq \overline{\mbox{Eff}}(X) \times \mathfrak{t}^*$ to be the closed convex cone generated by $J$.\\

\noindent (ii) Let $J_Y \subseteq \overline{\mbox{Eff}}(Y) \times \mathfrak{t}^*$ be the subset of all pairs $(D\mid_Y, \lambda)$ with $(D,\lambda)\in J$, and define $K_Y$ to be the closed convex cone generated by $J_Y$.
\end{defin}

\begin{lemma}\label{lemma:Kpolyhedral}
 The convex cone $K$ is rational polyhedral.
\end{lemma}

\begin{proof}
 Since $X$ is a Mori dream space, the Cox ring, $\mbox{Cox}(X)$, of $X$ is a finitely generated $\C$-algebra. Hence, the subalgebra $\mbox{Cox}(X)^U$ of unipotent invariants 
 is also finitely generated over $\C$ (cf. \cite[Theorem 16.2]{Gr}). We can then choose a set of generators $\zeta_1,\ldots, \zeta_N$ for $\mbox{Cox}(X)^U$ which are homogeneous with respect to the grading by divisors
 as well as with respect to the grading of the $U$-invariants by the weights, i.e., we can assume that $\zeta_i \in V(\lambda_i)^U \subseteq H^0(X, \O_X(D_i))^U$, for an effective divisor $D_i$ and 
 a weight $\lambda_i$. Moreover, we can assume that the restriction maps \eqref{E: equivrestr} are injective for $\lambda_1,\ldots, \lambda_\ell$, and zero for 
 $\lambda_{\ell+1},\ldots, \lambda_N$.
 
 Now, let $(D, \lambda)$ be a pair for which the map \eqref{E: equivrestr} is injective, and let $s \in V(\lambda) \subseteq H^0(X, \O_X(D))$ be a highest weight section, so $s\in V(\lambda)^U$.
 Writing $$s=\sum_{\underline{m}=(m_1,\ldots, m_N) \in \N_0^N} c_{\underline{m}} \zeta_1^{m_1} \cdots \zeta_N^{m_N}, \quad c_{\underline{m}} \in \C,$$ 
 we see 
 that the sum of the terms for $(m_1,\ldots, m_N)$ not satisfying $D=m_1D_1+\cdots +m_N D_N$ is zero. The same holds for the sum of the terms for $(m_1,\ldots, m_N)$ 
 not satisfying $\lambda=m_1\lambda_1+\cdots +m_N\lambda_N$. Hence, we can write $s$ as a linear combination of monomials in the $s_i$ for which 
 both conditions are satisfied. 
 
 We now claim that the sum $s=\sum_{\overline{m} \in \N_0^N} c_{\underline{m}}\zeta_1^{m_1}\cdots \zeta_N^{m_N}$ is even a linear combination of monomials only in 
 $\zeta_1,\ldots, \zeta_\ell$, that is, we have $m_{\ell+1}=\cdots=m_N=0$ for all $\underline{m}$ occurring in the sum. Indeed, if we had $m_i>0$ for some 
 $i \in \{\ell+1,\ldots, N\}$, the monomial $\zeta_1^{m_1}\cdots \zeta_N^{m_N}$ would restrict to zero on $Y$, so that 
 the restriction map \eqref{E: equivrestr} would be zero for the highest weight module $V(m_1\lambda_1+\cdots m_N \lambda_n) \subseteq H^0(X, \O_X(m_1D_1+\cdots + m_ND_N))=H^0(X, \O_X(D))$. 
 In particular, this would contradict the fact that $m_1\lambda_1+\cdots +m_N\lambda_N=\lambda$. 
 
 It follows that $K$ is the closed convex cone generated by the points $(D_1,\lambda_1),\ldots, (D_\ell, \lambda_\ell)$.
 \end{proof}
 
 \begin{defin}
  Let $\varphi\colon  N^1(X)_\R \oplus \mathfrak{t}^* \to N^1(Y)_\R \oplus \mathfrak{t}^*$ be the unique $\R$-linear map 
  satisfying $\varphi((D, \lambda))=(D\mid_Y, \lambda)$ for all pairs $(D, \lambda)$ where $D$ is a divisor and $\lambda\in \mathfrak{t}^*$.
 \end{defin}
 
 Lemma~\ref{lemma:Kpolyhedral} has the following immediate corollary.
 
 \begin{cor}
  The convex cone $K_Y$ is given by $K_Y=\varphi(K)$.
  In particular, $K_Y$ is rational polyhedral.
 \end{cor}

\begin{defin}
We define the semigroups 
 \begin{align}
  \widetilde J&:=\{(\nu(s), D, \lambda) \mid s \in V(\lambda)\setminus \{0\} \subseteq H^0(X, \O_X(D)), \; (D, \lambda) \in J\}, \\
  \widetilde J_Y&:=\{(\nu_Y(t), E, \lambda) \mid t \in V(\lambda) \setminus \{0\} \subseteq H^0(Y, \O_Y(E)), \; (E, \lambda) \in J_Y\},
 \end{align}
and the homomorphism of semigroups
\begin{align*}
 \widetilde\varphi\colon  \widetilde J \to \widetilde J_Y, \quad (\nu(s), D, \lambda) \mapsto (\nu_Y(s\mid_Y), D\mid_Y, \lambda).
\end{align*}
We also denote by $\widetilde K=\overline{\mbox{cone}}(\widetilde J)$ and $\widetilde K_Y=\overline{\mbox{cone}}(\widetilde J_Y)$ the closed convex cones generated by $\widetilde J$ and $\widetilde J_Y$, respectively. The map $\widetilde\varphi$ extends uniquely to a continuous map $\widetilde K\to \widetilde K_Y$, also denoted by $\widetilde\varphi$, for simplicity, between the respective closed convex cones.
\end{defin}

Since the convex cone $K$ is generated by pairs $(D, \lambda)$ for which the restriction map \eqref{E: equivrestr} is injective, the map $\widetilde\varphi$ between the convex cones is surjective, i.e., 
\begin{align*}
 \widetilde K_Y=\widetilde\varphi(\widetilde K).
\end{align*}

\begin{lemma} \label{L: interior}
 For any big divisor $D$ on $Y$ there exist a rational weight $\lambda$ and $a \in \Q^{n-1}$ such that $(a, D, \lambda)$ is 
 in the relative interior of $\widetilde \Delta_{\nu_Y}(Y)$.
\end{lemma}

\begin{proof}
 By the inductive assumption, the convex cone $\widetilde \Delta_{\nu_Y}(Y)$ is rational polyhedral. Let $\{(a_i, D_i, \lambda_i)\;|\; i\in\{1,\ldots,m\} \}$, for some $m \in \N$, be a set of 
 rational generators. We may assume that the rational divisors $D_1,\ldots, D_m$ generate the pseudoeffective cone $\overline{\mbox{Eff}}(Y)$ of $Y$. Note that we do not require the $\lambda_i$ nor the $D_i$ to be distinct here.
 
 Since $D$ is in the interior of $\overline{\mbox{Eff}}(Y)$, we can write $D$ as a linear combination $D=\xi_1 D_1+\cdots +\xi_m D_m$ with rational coefficients 
 $\xi_1,\ldots, \xi_m>0$.
 
 Put
\begin{align*}
a = \sum_{i=1}^m \xi_i a_i, \quad \lambda = \sum_{i=1}^m \xi_i \lambda_i.
\end{align*}
Then 
\begin{align*}
(a,D,\lambda)=\sum_{i=1}^m \xi_i (a_i, D_i, \lambda_i)
\end{align*}
is in the relative interior of $\widetilde \Delta_{\nu_Y}(Y)$.
\end{proof}

\begin{defin}
 Let $\pi\colon  \R^{n-1} \oplus N^1(Y)_\R \oplus \mathfrak{t}^* \to N^1(Y)_\R \oplus \mathfrak{t}^*$ denote the projection onto the 
direct summand $N^1(Y)_\R \oplus \mathfrak{t}^*$.
\end{defin}

\begin{lemma}\label{lemma:tKYpolyhedral}
If the convex cone $\widetilde \Delta_{\nu_Y}(Y)$ is rational polyhedral, then
 \begin{align*}
  \widetilde K_Y=\widetilde \Delta_{\nu_Y}(Y) \cap \pi^{-1}(K_Y).
 \end{align*}
In particular, the convex cone $\widetilde K_Y$ is rational polyhedral, being the intersection of two rational polyhedral convex cones.
\end{lemma}

\begin{proof}
Notice that $\widetilde J_Y=\widetilde \Delta_{\nu_Y}(Y)\cap \pi^{-1}(J_Y)$ and that the convex cone $\pi^{-1}(K_Y)$ is rational polyhedral and generated by $\pi^{-1}(J_Y)$.

We claim that $\pi^{-1}(J_Y)$ intersects the relative interior of $\widetilde \Delta_{\nu_Y}(Y)$. Given this claim, the lemma is a consequence of 
 \cite[Prop. A.1.]{LM09} (and its proof).
 
 In order to prove the claim, let $D$ be an ample divisor on $X$, and let $E:=D\mid_Y$. By Lemma \ref{L: interior}, there is a point of the form $(a, E, \lambda)$ in the relative interior of $\widetilde \Delta_{\nu_Y}(Y)$. By replacing 
 this triple by a positive integer multiple, if necessary, we may assume that $(a, E, \lambda)$ is an integral point in $\widetilde \Delta_{\nu_Y}(Y)$, and, by Serre's vanishing theorem, that the restriction map $R_{D}\colon  H^0(X, \O_X(D)) \to H^0(Y, \O_Y(E))$ of sections is surjective.
 
 By the $G$-equivariance of $R_D$, 
 the weight $\lambda$ occurs as the highest weight of an irreducible $G$-submodule $V(\lambda) \subseteq H^0(X, \O_X(D))$, and the restriction $R_{D}\mid_{V(\lambda)}\colon  V(\lambda) \to H^0(Y, \O_Y(E))$ is injective, i.e., $(E, \lambda) \in J_Y$. This finishes the proof.
\end{proof}

The above lemma enables us to deduce the rational polyhedrality of the convex cone $\widetilde K$.

\begin{lemma}
If the convex cone $\widetilde \Delta_{\nu_Y}(Y)$ is rational polyhedral, then
 \begin{align*}
 \widetilde K= ((\{0\} \times \R_{\geq 0}^{n-1}) \times K)
 \cap  \widetilde\varphi^{-1}(\widetilde K_Y) 
 \end{align*}
In particular, the convex cone $\widetilde K$ is rational polyhedral.
\end{lemma}

\begin{proof}
We use the same strategy as in Lemma~\ref{lemma:tKYpolyhedral}. If a point of the form $(\nu(s), D, \lambda)$ is in $\widetilde J$ for some $s \in V(\lambda)\setminus \{0\} \subseteq H^0(X, \O_X(D))$ and $(D, \lambda) \in J$, then $\nu(s)\in \{0\}\times \R^{n-1}$ by the definition of $J$.

This yields
 \begin{align*}
 \widetilde J=((\{0\} \times \R_{\geq 0}^{n-1}) \times J)  \cap \widetilde\varphi^{-1}(\widetilde J_Y).
 \end{align*}

We claim that $((\{0\} \times \N_{\geq 0}^{n-1}) \times J)$ intersects the relative interior of $\widetilde\varphi^{-1}(\widetilde K_Y)$. Given this claim, the lemma follows from \cite[Prop. A.1.]{LM09}.


To prove the claim, let $(a, E,\lambda)\in \widetilde J_Y$ be in the relative interior of $\widetilde K_Y$, choose $D$ such that $(D,\lambda)\in J$ and $E=D\mid_Y$, and choose $t\in V(\lambda)\setminus \{0\}\subset H^0(Y,\O_Y(E))$ such that $\nu_Y(t)=a$. Up to replacing $(a, E,\lambda)\in \widetilde J_Y$ by a positive integer multiple, we may assume that the restriction map $R_{D}\colon  H^0(X, \O_X(D)) \to H^0(Y, \O_Y(E))$ is surjective. Then we have $V(\lambda) \subseteq H^0(X, \O_X(D))$, and the map $R_{D}$ sends $V(\lambda)\subseteq H^0(X, \O_X(D))$ isomorphically to $V(\lambda)\subseteq H^0(Y, \O_Y(E))$. In other words, the section $t$ extends to a section $s\in V(\lambda)\subseteq H^0(X, \O_X(D))$.

It follows that $\nu(s)=(0,a)$, so $(\nu(s),D,\lambda)\in (\{0\} \times \N_{\geq 0}^{n-1}) \times J$ and $\widetilde\varphi(\nu(s),D,\lambda)=(t,E,\lambda)$, which finishes the proof.
\end{proof}

 We can now prove the main result, deducing the rational polyhedrality of $\widetilde \Delta_\nu(X)$ from that of $\widetilde \Delta_{\nu_Y}(Y)$. 
 
 \begin{lemma}\label{L:polyhedralinduct}
  If the convex cone $\widetilde \Delta_{\nu_Y}(Y)$ is rational polyhedral, then so is $\widetilde \Delta_\nu(X)$. 
 \end{lemma}

 \begin{proof}
  Let $$((0,\eta_1), D_1, \lambda_1), \ldots, ((0,\eta_m), D_m, \lambda_m) \in \{0\}\times \R_{\geq 0}^{n-1} \times K \subseteq \R^n \oplus N^1(X)_\R \oplus \mathfrak{t}^*$$ be integral generators of 
  the convex cone $\widetilde K$.\\
  
  Now, let $s \in V(\lambda) \subseteq H^0(X, \O_X(D))$ be a nonzero section. If $s$ vanishes to order $a$ along $Y$, we have that $s':=\frac{s}{s_1^a} \in V(\lambda) \subseteq H^0(X, \O_X(D-aY))$ 
  is a section with nonzero restriction to $Y$, so that
  $(\nu(\xi), D-aY, \lambda) \in \widetilde J$. Then 
  \begin{align*}
  (\nu(s), D, \lambda)&=(\nu(\xi), D-aY, \lambda)+((a,0,\ldots,0), aY, 0)\\
  &\in \mbox{cone}\{((0,\eta_1), D_1,\lambda_1),\ldots, ((0,\eta_m), D_m,\lambda_m), ((1,0,\ldots, 0), Y, 0)\}.
  \end{align*}
  It follows that 
  \begin{align*}
   \widetilde \Delta_\nu(X)=\mbox{cone}\{((0,\eta_1), D_1,\lambda_1),\ldots, ((0,\eta_m), D_m,\lambda_m), ((1,0,\ldots, 0), Y, 0)\}.
  \end{align*}

 \end{proof}

\begin{thm}\label{T:polyhedral}
The convex cones $\widetilde \Delta_\nu(X)$ and $\Delta_\nu(X)$ are rational polyhedral.
\end{thm}
\begin{proof}
We proceed by induction on $r$ (the number of the $G$-invariant prime Cartier divisors $E_1,\ldots, E_r$, cf.\ Section~\ref{s:setting}). If $r=0$ then 
$X=G/P$ is $G$-homogeneous, and $\Delta_\nu(X)$ is rational polyhedral by \cite[Corollary~4.11]{SS17}.

The inductive step follows from the fact that $Y=Y_1$ admits the partial flag of subvarieties $Y=Y_1\supset Y_2 \supset\cdots\supset Y_r$, from 
Lemma~\ref{L:polyhedralinduct}, and the fact that $\Delta_\nu(X)$ is the image of $\widetilde \Delta_\nu(X)$ under the projection $\R^{n} \oplus N^1(X)_\R \oplus \mathfrak{t}^*\to \R^{n} \oplus N^1(X)_\R$.
\end{proof}



\end{document}